\documentclass[12pt]{article}

\usepackage{amssymb}
\usepackage{amsthm}
\usepackage{amsmath}
\usepackage{graphicx}
\usepackage{fullpage}
\usepackage{color}
\allowdisplaybreaks
 \numberwithin{equation}{section}

\theoremstyle{plain}
\newtheorem{thm}{Theorem}[section]
\newtheorem{cor}[thm]{Corollary}

\newtheorem{lem}[thm]{Lemma}
\newtheorem{prop}[thm]{Proposition}

\theoremstyle{definition}
\newtheorem{defn}[thm]{Definition}

\theoremstyle{remark}

\newtheorem{rem}[thm]{Remark}

\newcommand{\N}{\mathbb{N}}

\newcommand{\R}{\mathbb{R}}

\newcommand{\E}{\mathbb{E}} 
\newcommand{\PP}{\mathbb{P}} 

\newcommand{\I}{\infty}
\newcommand{\diam}{\text{diam}}
\newcommand{\dist}{\text{dist}}
\newcommand{\Sn}{S_n(\R)}
\newcommand{\Mn}{M_n(\R)}
\newcommand{\tr}{\text{tr}}
\newcommand{\bp}{\begin{proof}[\ensuremath{\mathbf{Proof}}]}
\newcommand{\ep}{\end{proof}}
\newcommand{\beps}{\beta_\epsilon}
\newcommand{\bF}{\overline{F}}
\newcommand{\gfint}{\int_{B_r(x_0)} \!\!\!\!\!\!\!\!\!\!\!\!\!\!\!\!\!-}

\begin{document}

%Title

\title{On Hamilton-Jacobi-Bellman equations with convex gradient constraints}

\author{Ryan Hynd\footnote{Department of Mathematics, University of Pennsylvania. Partially supported by NSF grant DMS-1301628.}\; and 
Henok Mawi\footnote{Department of Mathematics, Howard University.}}  

\maketitle

%  Abstract  
\begin{abstract}
We study PDE of the form $\max\{F(D^2u,x)-f(x), H(Du)\}=0$ where $F$ is uniformly elliptic and convex in its first argument, $H$ is convex, $f$ is a given function and $u$ is the unknown. 
These equations are derived from dynamic programming in a wide class of stochastic singular control problems.  In particular, examples of these equations arise in mathematical finance models
involving transaction costs, in queuing theory, and spacecraft control problems. The main aspects of this work are to identify conditions  under which 
solutions are uniquely defined and have Lipschitz continuous gradients.  We also generalize previous results known for the case where $M\mapsto F(M,x)$ is the maximum of finitely many linear functions. 
\end{abstract}

%%%%%%%%%%%%%%%%%%%%%%%%%%%%%%%%%% Introduction %%%%%%%%%%%%%%%%%%%%%%%%%%%%%%%%%
\section{Introduction}
% History
In 1979, L.C. Evans inaugurated the study of elliptic equations with gradient constraints when he considered the following Dirichlet problem: find a function $u:\overline{\Omega}\rightarrow \R$ satisfying the PDE
\begin{equation}\label{EvansLinear}
\max\left\{-a(x)\cdot D^2u - f(x), |Du|-g(x)\right\}=0, \quad x\in \Omega
\end{equation}
subject to the boundary condition 
$$
 u(x)=0, \quad x\in \partial \Omega
$$
\cite{E}. Here $\Omega\subset \R^n$ is a bounded domain with smooth boundary, $f$ and $g$ are smooth positive functions on $\Omega$, and the symmetric $n\times n$ matrix valued function $a=(a^{ij})$ is smooth and uniformly elliptic. 
That is, there are positive numbers $\lambda,\Lambda$ for which
\begin{equation}\label{LinUnifEll}
\lambda|\xi|^2\le a(x)\xi\cdot \xi\le \Lambda|\xi|^2, \quad x\in\overline{\Omega},\quad \xi\in \R^n. 
\end{equation}
In equation \ref{EvansLinear}, we have used the notation $Du=(u_{x_i})$, $D^2u=(u_{x_ix_j})$ and
$$
a(x)\cdot D^2u:=\tr(a(x)D^2u)=a^{ij}(x)u_{x_ix_j}. 
$$
\par The PDE \eqref{EvansLinear} is considered an elliptic equation with gradient constraint as solutions necessarily satisfy
$$
|Du|\le g(x), \quad x\in \Omega.
$$
Moreover, on the subset of $\Omega$ where the above inequality is strict, solutions satisfy the elliptic PDE
$$
-a(x)\cdot D^2u = f(x).
$$
As we will see below, equations such as \eqref{EvansLinear} are naturally interpreted as dynamic programming equations in the theory of stochastic singular control. 

\par Employing Bernstein's method, L.C. Evans showed that equation \eqref{EvansLinear} has a unique solution that satisfies the PDE \eqref{EvansLinear} almost everywhere and belongs to the space $W^{2,p}_\text{loc}(\Omega)\cap W_0^{1,\infty}(\Omega)$ for each $p\in [1,\infty)$. Furthermore, if the coefficient matrix $a=(a^{ij})$ is {\it constant}, then additionally $u\in W^{2,\infty}_\text{loc}(\Omega)$.  Shortly thereafter, M. Wiegner removed the requirement that the coefficient matrix $a=(a^{ij})$ is constant and derived an a priori $W^{2,\infty}_\text{loc}(\Omega)$ estimate on solutions \cite{W}. Finally, H. Ishii and S. Koike considered a version of the
above Dirichlet problem involving more general gradient constraints and verified solutions belong to the space $W^{2,\infty}(\Omega)$; these authors also showed with simple examples that if $f$ and $g$ are allowed to 
vanish simultaneously, uniqueness of solutions may fail. We also remark that H.M. Soner and S. Shreve studied closely related equations in two variables \cite{SS} and equations that involved special structure \cite{SS1, SS2}; and in both cases they verified the existence of classical solutions.

% Yamada 
\par  M. Wiegner's regularity result was further extended by N. Yamada who considered the Dirichlet problem associated with the Bellman equation 
\begin{equation}\label{YamadaEq}
\max_{1\le k\le N}\left\{-a_k(x)\cdot D^2u - f(x), |Du|-g(x) \right\}=0, \quad x\in \Omega,
\end{equation}
where each $a_k$ satisfies \eqref{LinUnifEll}. N. Yamada used a clever argument (inspired by previous work of L.C. Evans and A. Friedman \cite{EF}) to verify the existence of a solution $u\in W^{2,\infty}_{\text{loc}}(\Omega)\cap W^{1,\infty}_0(\Omega)$ \cite{Y}.  To date, this is the best regularity result for nonlinear elliptic equations with gradient constraints.

%Main PDE/Assumption
\par In this paper, we develop the regularity theory further by considering viscosity solutions of a fully nonlinear analog of \eqref{EvansLinear} 
\begin{equation}\label{MainPDE}
\max\{F(D^2u,x)-f(x), H(Du)\}=0, \quad x\in \Omega.
\end{equation}
Observe the particular nonlinearity 
\begin{equation}\label{YamadaNon}
F(M,x)= \max_{1\le k\le N}\left\{-a_k(x)\cdot M\right\}
\end{equation}
and gradient constraint $H(p)=|p|-1$ correspond to \eqref{YamadaEq} provided $g\equiv 1$ (and likewise to \eqref{EvansLinear} when $N=1$).  Postponing the definition of viscosity solutions until the next 
section (see Definition \ref{ViscDefn}), let us first discuss the relevant structural conditions needed on $F$ to guarantee a reasonable theory associated to the PDE \eqref{MainPDE}.

Denoting $\Sn$ as the collection of real, symmetric $n\times n$ matrices, we assume that $F\in C(\Sn\times\overline{\Omega})$ and is uniformly elliptic in 
its first argument. That is, there are positive numbers $\lambda,\Lambda$ such that
\begin{equation}\label{FUnifEll}
-\Lambda \tr N\le F(M+N,x)- F(M,x)\le -\lambda \tr N
\end{equation}
for each $M, N\in\Sn$ with $N\ge 0$ and $x\in \overline{\Omega}$.  We also suppose throughout that $F$ is convex in its first argument
\begin{equation}\label{Fconvex}
F(sM+(1-s)N,x)\le sF(M,x)+(1-s)F(N,x),\quad M,N\in\Sn, \; s\in [0,1]
\end{equation}
and that there is $\Upsilon> 0$ for which
\begin{equation}\label{FXassump}
|F(M,x) - F(M,y)|\le \Upsilon (|M|+1)|x-y|, \quad x,y\in \overline{\Omega},\; M\in\Sn. 
\end{equation}
In \eqref{FXassump}, $|M|=\sqrt{\sum^n_{i,j=1}M_{ij}^2}$.
% Main Theorem 
\par The main theorem of this work is as follows. 
\begin{thm}\label{mainTHM}
Let $\Omega\subset \R^n$ be a bounded domain with smooth boundary, $\varphi\in C(\partial \Omega)$ and $f\in C(\overline{\Omega})$. 
\\ (i) Assume that $H$ is convex and that there is $\underline{u}\in C^2(\Omega)\cap C(\overline{\Omega})$ satisfying 
\begin{equation}\label{ubarEqn}
\begin{cases}
\max\{F(D^2\underline{u},x)-f(x), H(D\underline{u})\}\le -\kappa , \quad x\in \Omega\\
\hspace{2.2in} \underline{u}=\varphi, \quad x\in \partial \Omega
\end{cases}.
\end{equation}
for some $\kappa>0$.  Moreover, suppose $F$ satisfies \eqref{TechAssF} below. Then there is a unique viscosity solution $u\in C(\overline{\Omega})$ of the PDE \eqref{MainPDE} subject to the boundary condition
\begin{equation}\label{newBC}
u(x)=\varphi(x), \quad x\in \partial \Omega.
\end{equation}
(ii) Assume further that there are positive numbers $\theta,\Theta$ such that $H$ satisfies  
\begin{equation}\label{Hassump}
\theta |\xi|^2\le D^2H(p)\xi\cdot \xi\le \Theta |\xi|^2, \quad  \xi\in \R^n
\end{equation}
for Lebesgue almost every $p\in \R^n$ and $f\in W^{1,\infty}(\Omega)$. Then 
$$
u\in W^{2,p}_\text{loc}(\Omega)\cap W^{1,\infty}(\Omega)
$$
for each $p\in [1,\infty)$.  \\
(iii) Additionally, if $f\in W^{2,\infty}(\Omega)$ and $F$ is independent of $x$, then 
$$
u\in W^{2,\infty}_\text{loc}(\Omega).
$$ 
\end{thm}
% Remarkes
\noindent Theorem \ref{mainTHM} is the first to address the regularity of solutions of general fully nonlinear elliptic equations with convex gradient constraints. However, 
we acknowledge that it does not improve N. Yamada's result \cite{Y} for the particular nonlinearity \eqref{YamadaNon}. While we do not have a counterexample, we have 
identified a clear technical obstruction to removing the assumption that $F$ is independent of $x$ in order to obtain an a priori $W^{2,\infty}_\text{loc}(\Omega)$ estimate.  See Remark \ref{WiegRem} below. 

\par This work also generalizes the first author's previous work \cite{Hynd}, which he considered equation \eqref{MainPDE} with $F(M,x)=-a(x)\cdot M$. In \cite{Hynd}, an a priori $W^{2,\infty}_{\text{loc}}(\Omega)$ estimate was derived on solutions under the assumption \eqref{Hassump}. 
However, an a priori $W^{2,p}_{\text{loc}}(\Omega)$ estimate was claimed to be obtained for arbitrary convex gradient constraint functions $H$. Upon further review, we now believe that the asserted $W^{2,p}_{\text{loc}}(\Omega)$  estimate requires a uniform convexity hypothesis similar to \eqref{Hassump}. 

\par It should also be noted that for a given convex gradient constraint function $H$,  if $\{H\le 0\}=\{G\le 0\}$ then the PDE \eqref{MainPDE} holds with $G$ replacing $H$. For instance if 
$H(p)=|p|-1$, we may use $G(p)=|p|^2-1$ which additionally satisfies \eqref{Hassump}. In view of the statement of Theorem \ref{mainTHM}, it is advantageous to work with uniformly convex gradient constraints when they are available. We also
acknowledge that we only consider gradient constraints that do not depend on the $x$ variable. However, readers may verify without much difficulty that Theorem \ref{mainTHM} holds for convex gradient
constraints  $H=H(p,x)$ that satisfy $\theta|\xi|^2\le D^2_pH(p,x)\xi\cdot \xi\le \Theta |\xi|^2$. 

% Organization
\par In the work that follows, we will establish Theorem \ref{mainTHM} in a series of steps.  In section \ref{SecComp}, we verify a comparison result which proves part $(i)$. In sections, \ref{SecPen} and \ref{SecConv} we introduce  a penalized equation and derive some uniform estimates that will imply parts $(ii)$ and $(iii)$. Before proceeding to the proof of Theorem \ref{mainTHM}, let us first give a brief motivation of how equation \eqref{MainPDE} is derived in stochastic control theory. We refer readers to standard references such as Chapter VIII of \cite{FS} or Chapter 5 of \cite{Oks} for the necessary background material.  And for applications of stochastic singular control theory, we recommend \cite{BC,BSW} (spacecraft control) \cite{DN, DPZ} (mathematical finance) and \cite{Kush} (queueing theory).

% ------------------------------------ Probability ------------------------------------
\par {\bf Probabilistic interpretation of solutions}.  Let $(\Omega, {\cal F}, \PP)$ be a probability space with a standard $n$-dimensional Brownian motion $(W(t), t\ge 0)$. For a fixed set $U\subset \R^m$ $(m\in \N)$, we define a {\it control process} to be a triple $(\alpha, \rho,\xi)$ such that 
$$
\begin{cases}
(\alpha(t),\rho(t), \xi(t))\in U\times \R^n\times\R \\
(\alpha,\rho, \xi)\; \text{is adapted to the filtration generated by $W$}\\
|\rho(t)|=1, \; t\ge 0\;\; \text{$\PP$ almost surely}\\
\xi(0)=0\quad \text{$\PP$ almost surely}\\
t\mapsto \xi(t),\text{is non-decreasing, and is left continuous and has right limits $\PP$ almost surely}
\end{cases}.
$$
Associated to any control is an $\R^n$-valued diffusion process $(X^{\alpha,\rho, \xi})$ that satisfies the stochastic differential equation
\begin{equation}\label{ControlSDE}
\begin{cases}
dX(t)=\sigma(X(t), \alpha(t))dW(t) - \rho(t)d\xi(t)\quad t\ge 0\\
X(0)=x\in \R^n
\end{cases}.
\end{equation}
Here $\sigma: \R^n\times U\rightarrow \Mn$ is assumed to be continuous, where $\Mn$ is the collection of all real $n\times n$ matrices. We also suppose there is $L>0$ such that
\begin{equation}\label{LipschitzSig}
|\sigma(x,z)-\sigma(y,z)|\le L|x-y|, \quad x,y\in\overline{\Omega} 
\end{equation}
for each $z\in U$.  We note that under assumption \eqref{LipschitzSig}, \eqref{ControlSDE} has a solution for each control $(\alpha,\rho, \xi)$.

\par Recall that for a nonempty, closed, convex set $K\subset\R^n$, the corresponding support function is 
given by
$$
\ell(v)=\sup_{p\in K}v\cdot p, \quad v\in \R^n. 
$$
The optimization problem we are most interested in involves the following {\it value function}
\begin{equation}\label{ValueFun}
u(x):=
\inf_{\alpha,\rho, \xi}\E\int^\tau_0\left[f(X^{\alpha,\rho, \xi}(t))dt + \ell(\rho(t))d\xi(t)\right], \quad x\in \overline{\Omega}.
\end{equation}
Here $\tau:=\inf\{t\ge 0: X^{\alpha,\rho, \xi}(t)\notin\Omega \}$.  The problem of finding an optimal control process for $u(x)$ is one of {\it stochastic singular control}. This terminology is used as 
a typical control process $(\alpha, \rho,\xi)$ involves $\xi$ which may have samples paths that are not everywhere continuous. 

\par We will argue that the value function \eqref{ValueFun} satisfies a PDE of the form \eqref{MainPDE}.  To see this, we first consider the related value function 
$$
u^N(x):=
\inf_{\alpha,\rho, \gamma}\E\int^\tau_0\left[f(X^{\alpha,\rho, \gamma}(t))dt + \ell(\rho(t))\gamma(t)dt\right], \quad x\in \overline{\Omega}.
$$
Here $(X^{\alpha,\rho, \gamma})$ satisfies \eqref{ControlSDE} with $\xi(t)=\int^t_0\gamma(s)ds$, for a process $\gamma$ adapted to the filtration generated by $W$ with sample paths $[0,\infty)\ni t\mapsto\gamma(t)\in [0,N]$. Note the value function $u^N$ corresponds to 
a standard stochastic optimal control problem as the controlled diffusions $(X^{\alpha,\rho, \gamma})$ are pathwise continuous almost surely. 

\par In particular, $u^N$ is known to formally satisfy the Hamilton-Jacobi-Bellman equation 
\begin{align*}
0 & =\sup_{\substack{z\in U, |v|=1\\ 0\le r\le N}}\left\{-\frac{1}{2}\sigma(x,z)\sigma(x,z)^t\cdot D^2u^N-f(x) + r\left(Du^N\cdot v -\ell(v)\right) \right\}\\
  & = F(D^2u^N,x)-f(x)+N\left[H(Du^N)\right]^+
\end{align*}
where
\begin{equation}\label{HJBF}
F(M,x):=\sup_{z\in U}\left\{-\frac{1}{2}\sigma(x,z)\sigma(x,z)^t\cdot M\right\}
\end{equation}
and 
\begin{equation}\label{HJBH}
H(p):=\sup_{|v|=1}\left\{p\cdot v -\ell(v)\right\}
\end{equation}
(see VIII.2 of \cite{FS}).

\par Observe that $F(D^2u^N,x)-f(x)\le 0$, and for large $N$, we also expect $H(Du^N)\le 0$. Of course, this is a heuristic argument and only applies to the value function $u^N$. Nevertheless, it is possible 
to show rigorously that the value function $u$ satisfies \eqref{MainPDE} with $F$ given by \eqref{HJBF} and $H$ given by \eqref{HJBH} in the sense of viscosity solutions provide $u$ itself satisfies a dynamic programming principle. The following theorem details this connection; 
it's proof, however, will be omitted as the assertion follows from a straightforward generalization of Theorem 5.1 in section VIII of \cite{FS}. We also remind the reader that we will postpone a discussion of viscosity solutions until the next section. 

\begin{prop}
Assume that for each $x\in \Omega$ and each stopping time $T$ (with respect to the filtration generated by $W$), the value function $u$ satisfies the dynamic programming principle 
\begin{align*}
u(x)&=\inf_{\alpha,\rho, \xi}\E\left\{\int^{\tau\wedge T}_0\left[f(X^{\alpha,\rho, \xi}(t))dt + \ell(\rho(t))d\xi(t)\right] + u(X^{\alpha,\rho, \xi}(\tau\wedge T)) \right\}.
\end{align*}
Then $u$ is a viscosity solution of the Hamilton-Jacobi-Bellman equation \eqref{MainPDE} with $F$ given by \eqref{HJBF} and $H$ given by \eqref{HJBH}. 
\end{prop}
% Tie in with our main result 
Combining this proposition with the Theorem \ref{mainTHM} gives the following. 
\begin{cor}
Assume $f>0$ on $\overline{\Omega}$, 
$$
\frac{1}{2}\sigma(\cdot,z)\sigma(\cdot ,z)^t\; \text{satisfies \eqref{LinUnifEll} for each $z\in U$}
$$
and 
$$
K:=\{p\in \R^n: G(p)\le 0\}
$$
where $G$ satisfies \eqref{Hassump} and $G(0)<0$. Then the value function $u$ is the unique solution of \eqref{MainPDE} that satisfies $u|_{\partial\Omega}=0$, 
with $F$ given by \eqref{HJBF} and $H$ given by \eqref{HJBH}.  Moreover, $u\in W^{2,p}_{\text{loc}}(\Omega)\cap W^{1,\infty}_0(\Omega)$ for each $p\in [1,\infty)$.
Finally, if $\sigma$ is independent of $x$, $u\in W^{2,\infty}_{\text{loc}}(\Omega)$. 
\end{cor}
\begin{proof}
By assumption, $H(0)<0$ and $\inf_\Omega f>0$. Therefore, $\underline{u}\equiv 0$ is a subsolution of \eqref{MainPDE} with $\kappa:=\max\{-\inf_\Omega f,H(0)\}$. 
Part $(i)$ of Theorem \ref{mainTHM} implies $u$ is the unique solution of \eqref{MainPDE} with $u|_{\partial\Omega}=0$.  Since $G\le 0$ if and only if $H\le 0$, $u$ satisfies  \eqref{MainPDE} with $G$ replacing $H$. By 
part $(ii)$ of Theorem \ref{mainTHM},  $u\in W^{2,p}_{\text{loc}}(\Omega)\cap W^{1,\infty}_0(\Omega)$ for each $p\in [1,\infty)$.  Likewise, if $\sigma$ is 
independent of $x$, then $F$ doesn't depend on $x$ and we conclude by part $(iii)$ of Theorem \ref{mainTHM}.
\end{proof}

%%%%%%%%%%%%%%%%%%%%%%%%%%%%%%%%%% Comparison %%%%%%%%%%%%%%%%%%%%%%%%%%%%%%%%%
\section{Comparison}\label{SecComp}
% Formal argument 
In this section, we will verify part $(i)$ of Theorem \ref{mainTHM}. In particular, we assume throughout this section that $H$ is convex and $\underline{u}$ satisfies \eqref{ubarEqn}. 
Our main assertion is that a comparison principle holds among viscosity sub- and supersolutions of PDE \eqref{MainPDE}. This fact is not obvious given that the equation is neither uniformly elliptic or proper.  What is interesting in this fully nonlinear framework is that the
convexity assumptions on $F$ and $H$ play a central role in this comparison principle. In particular, these are not merely assumptions to guarantee more regularity of solutions; these assumptions are also needed to verify the existence of a solution. Below, we will make use of the results and notation of the ``user guide"\cite{CIL}. First, let us recall the definition of viscosity sub- and supersolutions. 

\begin{defn}\label{ViscDefn}
A function $u\in USC(\Omega)$ is a {\it viscosity subsolution} of \eqref{MainPDE} if, whenever $\varphi\in C^2(\Omega)$ and $u-\varphi$ has a local maximum at $x_0\in \Omega$, 
$$
\max\{F(D^2\varphi(x_0),x_0)-f(x_0), H(D\varphi(x_0))\}\le 0.
$$
A function $v\in LSC(\Omega)$ is a {\it viscosity supersolution} of \eqref{MainPDE} if, whenever $\psi\in C^2(\Omega)$ and $u-\psi$ has a local minimum at $x_0\in \Omega$, 
$$
\max\{F(D^2\psi(x_0),x_0)-f(x_0), H(D\psi(x_0))\}\ge 0.
$$
A function $w\in C(\Omega)$ is a {\it viscosity solution} of \eqref{MainPDE} provided $w$ is a viscosity sub- and supersolution. 
 \end{defn}

\begin{prop}\label{CompareProp}
Suppose $u\in USC(\overline{\Omega})$ is a viscosity subsolution of \eqref{MainPDE}, $v\in LSC(\overline{\Omega})$ is a viscosity supersolution \eqref{MainPDE}, and $u\le v$ on $\partial \Omega$. Then $u\le v$.  
\end{prop}
Before proving Proposition \ref{CompareProp}, we give a heuristic argument as to why we would expect this result to be true. We suppose 
$u,v\in C^2(\Omega)\cap C(\overline{\Omega})$ $\tau\in (0,1)$ and set 
$$
w^\tau:=\tau u +(1-\tau)\underline{u} - v.
$$
Now assume $w^\tau(x_0)=\max_{\overline{\Omega}}w^\tau$. If $x_0\in \Omega$, 
\begin{equation}\label{WtauCalc}
\begin{cases}
Dw^\tau(x_0)=\tau Du(x_0) + (1-\tau)D\underline{u}(x_0) - Dv(x_0)=0,\\
D^2w^\tau(x_0)=\tau D^2u(x_0) + (1-\tau)D^2\underline{u}(x_0) - D^2v(x_0)\le 0
\end{cases}.
\end{equation}
By the convexity of $H$,
\begin{align*}
H(Dv(x_0)) &=H\left(\tau Du(x_0)+(1-\tau)D\underline{u}(x_0)\right)\\
& \le \tau H(Du(x_0))+(1-\tau)H(D\underline{u}(x_0))\\
& \le (1-\tau)H(D\underline{u}(x_0))\\
&<0.
\end{align*}
And since $v$ is a supersolution, 
\begin{equation}\label{YesVsupsoln}
f(x_0)\le F(D^2v(x_0),x_0).
\end{equation}
While we always have by the convexity of $F$
\begin{align*}
F(D^2(\tau u +(1-\tau)\underline{u})(x_0),x_0)&\le \tau  F(D^2u(x_0),x_0)+(1-\tau)F(D^2\underline{u}(x_0),x_0)\\
& < \tau f(x_0)+(1-\tau)f(x_0)\\
& = f(x_0).
\end{align*}
Now by \eqref{WtauCalc} and the hypothesis that $F$ is elliptic 
\begin{equation}\label{NoVsupsoln}
F(D^2v(x_0),x_0)\le F(\tau D^2u(x_0) + (1-\tau)D^2\underline{u}(x_0))<f(x_0).
\end{equation}
As inequalities \eqref{YesVsupsoln} and \eqref{NoVsupsoln} are incompatible, it must be that $x_0\in \partial \Omega$.  In this case, 
\begin{align*}
\tau u +(1-\tau)\underline{u} - v&=w^\tau \\
&\le w^\tau(x_0)\\
&=\tau u(x_0) +(1-\tau)\underline{u}(x_0)-v(x_0) \\
&\le \tau v(x_0) +(1-\tau)\underline{u}(x_0)-v(x_0)\\
&= (1-\tau)(\underline{u}(x_0)-v(x_0))\\
&\le (1-\tau)\max_{\overline{\Omega}}(\underline{u}-v)
\end{align*}
Sending $\tau\rightarrow 1^-$ gives, $u\le v$. 
% Technical assumption 
\par Now we proceed to the general situation. When studying the comparison of viscosity sub and super solutions of the elliptic PDE
$$
F(D^2u,x)=f(x)
$$
the following technical assumption is typically made: there is a function $\omega: [0,\infty)\rightarrow [0,\infty)$ that satisfies $\omega(0+)=0$ such 
that
\begin{equation}\label{TechAssF}
F(Y,y)- F(X,x)\le \omega\left(\frac{|x-y|^2}{\eta}\right)
\end{equation}
whenever 
\begin{equation}\label{IshiiCond}
\left(\begin{array}{cc}
X & 0 \\
0 & -Y
\end{array}
\right)\le\frac{3}{\eta}
\left(\begin{array}{cc}
I_n & -I_n \\
-I_n & I_n
\end{array}
\right)
\end{equation}
for $\eta>0$, $x,y\in \Omega$ and $X,Y\in\Sn$. See section 3 of \cite{CIL}. For instance, when $F(M,x)=-a(x)\cdot M$ with $a$ satisfying \eqref{LinUnifEll}, one may choose 
$$
\omega(r)=3\left(\text{Lip}(a^{1/2})\right)^2r.
$$
See example 3.6 in \cite{CIL}. Moreover, when $F$ is given by \eqref{HJBF}, we can take
$$
\omega(r)=\frac{3}{2}L^2r.
$$
% Viscosity Argument 
\begin{proof}(of Proposition \ref{CompareProp}) Fix $\tau\in (0,1)$, define 
$$
w^{\tau,\eta}(x,y):=\tau u(x) +(1-\tau)\underline{u}(x)- v(y)-\frac{1}{2\eta}|x-y|^2
$$
for $x,y\in\overline{\Omega}$ and $\eta>0$. By assumption, $w^{\tau,\eta}\in USC(\overline{\Omega}\times \overline{\Omega})$ so this function 
has a joint maximum at some $(x_\eta,y_\eta)\in \overline{\Omega}\times \overline{\Omega}$. By Lemma 3.1 in \cite{CIL}, 
$$
\lim_{\eta\rightarrow 0^+}\frac{|x_\eta-y_\eta|^2}{\eta}=0.
$$
The same lemma also asserts the existence of a sequence of $\eta$ tending to zero such that $(x_\eta,y_\eta)\rightarrow (x_\tau,x_\tau)$ and 
$x_\tau$ is maximizer of $\tau u +(1-\tau)\underline{u} - v$. If $x_\tau
\in \partial\Omega$, we obtain the estimate
%\begin{
$$
\tau u +(1-\tau)\underline{u} - v\le (1-\tau)\max_{\overline{\Omega}}(\underline{u}-v)
$$
as above and send $\tau\rightarrow 1^{-}$ to conclude $u\le v$.  Let us now assume $x_\tau\in \Omega$ and without loss of generality 
$x_\eta,y_\eta\in \Omega$ for all $\eta>0$.

\par By the Crandall-Ishii Lemma (Theorem 6.1, Chapter V \cite{FS}) that there are $X,Y\in \Sn$ such that \eqref{IshiiCond} holds
and 
$$
\begin{cases}
\left(\frac{x_\eta-y_\eta}{\eta},X\right)\in \overline{J}^{2,+}(\tau u + (1-\tau)\overline{u})(x_\eta)\\
\left(\frac{x_\eta-y_\eta}{\eta},Y\right)\in \overline{J}^{2,-}v(y_\eta)
\end{cases}.
$$
Note that as $\underline{u}\in C^2(\Omega)$,
$$
\left(\frac{x_\eta-y_\eta}{\tau\eta}-\frac{(1-\tau)}{\tau}D\underline{u}(x_\eta),\frac{1}{\tau}X- \frac{(1-\tau)}{\tau}D^2\underline{u}(x_\eta)\right)
\in \overline{J}^{2,+}u(x_\eta).
$$
And by the convexity of $H$ 
\begin{align*}
H\left(\frac{x_\eta-y_\eta}{\eta}\right)&=H\left(\tau\left(\frac{x_\eta-y_\eta}{\tau\eta}-\frac{(1-\tau)}{\tau}D\underline{u}(x_\eta)\right)+(1-\tau)D\underline{u}(x_\eta)\right)\\
&\le \tau H\left(\frac{x_\eta-y_\eta}{\tau\eta}-\frac{(1-\tau)}{\tau}D\underline{u}(x_\eta)\right)+(1-\tau)H(D\underline{u}(x_\eta))\\
&\le (1-\tau)H(D\underline{u}(x_\eta))\\
&<0.
\end{align*}
Since $v$ is a supersolution of \eqref{MainPDE}, we have
$$
0\le F(Y,y_\eta) -f(y_\eta).
$$
\par Also notice that as $F$ is convex
\begin{align*}
F\left(X,x_\eta\right)&=F\left(\tau\left(\frac{1}{\tau}X- \frac{(1-\tau)}{\tau}D^2\underline{u}(x_\eta),x_\eta\right)+(1-\tau)D^2\underline{u}(x_\eta),x_\eta\right)\\
&\le \tau F\left(\frac{1}{\tau}X- \frac{(1-\tau)}{\tau}D^2\underline{u}(x_\eta),x_\eta\right)+(1-\tau)F\left(D^2\underline{u}(x_\eta),x_\eta\right)\\
&\le \tau f(x_\eta)+(1-\tau)F\left(D^2\underline{u}(x_\eta),x_\eta\right)\\
&= f(x_\eta) + (1-\tau)\left[F\left(D^2\underline{u}(x_\eta),x_\eta\right)-f(x_\eta)\right]\\
&\le f(x_\eta) - (1-\tau)\kappa.
\end{align*}
Therefore, 
\begin{align*}
(1-\tau)\kappa & \le f(x_\eta)-F\left(X,x_\eta\right) \\
&\le f(x_\eta)-f(y_\eta) + F(Y,y_\eta)-F\left(X,x_\eta\right)\\
&\le f(x_\eta)-f(y_\eta) +\omega\left(\frac{|x_\eta-y_\eta|^2}{\eta}\right).
\end{align*}
However, sending $\eta\rightarrow 0$ along an appropriate sequence gives a contradiction. Thus, $x_\tau\in\partial \Omega$ and the assertion follows. 
\end{proof}

\par Under the assumptions \eqref{FUnifEll}, \eqref{Fconvex}, and \eqref{FXassump}, the boundary value problem
$$
\begin{cases}
F(D^2u,x)=f(x), \quad x\in \Omega\\
\hspace{.64in} u=\varphi, \quad x\in \partial \Omega
\end{cases}
$$
has a unique classical solution $\overline{u}\in C^{2}(\Omega)\cap C(\overline{\Omega})$; see Theorem 17.17 and exercise 17.4 of \cite{GT}. This follows from the ``continuity method" and the celebrated Evans-Krylov a priori estimates for convex, fully nonlinear elliptic equations
\cite{EvansC2, KrylovC2}. And applying Perron's method, as detailed in \cite{IshiiPerron,CIL}, it is easily verified that
$$
u(x):=\sup\left\{w(x): w\; \text{is a viscosity subsolution of \eqref{MainPDE} with}\; \underline{u}\le w\le \overline{u}\right\}
$$ 
is the unique viscosity solution of \eqref{MainPDE} satisfying the boundary condition \eqref{newBC}. This concludes the proof of 
part $(i)$ of Theorem \ref{mainTHM}.

%%%%%%%%%%%%%%%%%%%%%%%%%%%%%%%%%% Penalty %%%%%%%%%%%%%%%%%%%%%%%%%%%%%%%%%
\section{Penalty method}\label{SecPen}
% Penalty method
When studying the regularity of solutions \eqref{MainPDE}, it will be useful for us to differentiate 
the nonlinearity $F$ which is defined on $\Sn\times \Omega$. In order to conveniently do calculus on $F$ and to approximate $F$ with smooth nonlinearities, we will extend $F$
to a function defined on $\Mn\times \Omega$. The particular extension of $F$ we will employ is 
$$
\overline{F}(M,x):=F\left(\frac{1}{2}(M+M^t),x\right), \quad (M,x)\in \Mn\times \Omega. 
$$
The following lemma, stated without proof, asserts $\bF$ extends $F$ in a way that preserves the essential properties
of $F$. 
\begin{lem}
Assume \eqref{FUnifEll}, \eqref{Fconvex} and \eqref{FXassump}. Then $\bF$ satisfies
\begin{equation}\label{barFXassump}
\begin{cases}
-\Lambda \tr N\le \overline{F}(M+N,x)- \overline{F}(M,x)\le -\lambda \tr N \quad (N\ge 0)\\\\
\bF(sM+(1-s)N,x)\le s\bF(M,x)+(1-s)\bF(N,x)\\\\
\left|\bF(M,x) - \bF(M,y)\right|\le \Upsilon (|M|+1)|x-y| 
\end{cases}
\end{equation}
for $M,N\in\Mn$, $x,y\in \overline{\Omega}$, $s\in [0,1]$. 
\end{lem}
Therefore, we will identify the nonlinearity $F$ with its extension $\bF$ and assume without any loss of generality that $F$ is a function on $\Mn\times \Omega$. This 
is an assumption we will make for the remainder of this work. We also record that if $F$ is smooth,  \eqref{barFXassump} implies 
\begin{equation}\label{barFXassump2}
\begin{cases}
-\Lambda \tr N\le F_{M_{ij}}(M,x)N_{ij} \le -\lambda \tr N \quad (N\ge 0)\\\\
F_{M_{ij}M_{kl}}(M,x)N_{ij}N_{kl}\ge 0\\\\\
\left|F_{x_i}(M,x)\right|\le \Upsilon (|M|+1), \quad i=1,\dots, n 
\end{cases}
\end{equation}
for $M,N\in\Mn$, $x\in \Omega$.

We now consider the regularity of solutions \eqref{MainPDE}.   Following the work of L. C. Evans, we will employ the 
{\it penalty method}. That is we trade in the highly degenerate PDE with constraint \eqref{MainPDE}, for the family 
of approximating uniformly elliptic PDEs
\begin{equation}\label{penalized}
\begin{cases}
F(D^2u^\epsilon,x)+\beta_\epsilon(H(Du^\epsilon))=f(x), \quad x\in \Omega \\
\hspace{1.67in} u^\epsilon=\varphi, \hspace{.375in} x\in \partial\Omega
\end{cases}.
\end{equation}
Here the family $\{\beta_\epsilon\}_{\epsilon\in (0,1)}$ is assumed to satisfy 
% beta
\begin{equation}\label{betaAss}
\begin{cases}
\beta_\epsilon\in C^\infty(\R)\\
\beta_\epsilon(z)=0, \quad z\le 0\\
\beps(z)>0, \quad z>0\\
\beps'\ge0,\\
\beps''\ge0,\\
\beps(z)=(z-\epsilon)/\epsilon, \quad z\ge 2\epsilon\\
\end{cases}.
\end{equation}
We think of $\beta_\epsilon$ as a smooth approximation of $z\mapsto (z/\epsilon)^+$. Our goal is to obtain estimates of
solutions $u^\epsilon$ that are independent of $\epsilon$ in order control the term $\beta_\epsilon(H(Du^\epsilon))$.
Our intuition is that if this term is bounded independently of $\epsilon$, \eqref{betaAss} would force $H(Du^\epsilon)$ to become 
nonpositive as $\epsilon$ tends to $0$. 
%Regularity assumption/Classical Existence 
\par We will also initially assume 
\begin{equation}\label{ExtraSmooth}
\begin{cases}
F\in C^\I(\Mn\times \Omega)\\
H\in C^\I(\R^n)\\
f\in C^\I(\Omega)
\end{cases}
\end{equation}
and later remove this assumption with an approximation argument. Observe that the nonlinearity in \eqref{penalized} is 
$$
F^\epsilon(M,p,x):=F(M,x)+\beta_\epsilon(H(p)).
$$
As $H$ satisfies \eqref{Hassump}, for each fixed $\epsilon>0$  
$$
|F^\epsilon(M,p,x)|\le C_\epsilon(1+|M|+|p|^2)
$$
for some $C_\epsilon>0$.  Also observe that for each $p\in \R^n$, $(M,x)\mapsto F^\epsilon(M,p,x)$ satisfies \eqref{barFXassump}. By the work of N. Trudinger on 
classical solutions to fully nonlinear equations with ``natural structure conditions", the boundary value problem \eqref{penalized}
has a unique classical solution $u^\epsilon\in C^\infty(\Omega)\cap C(\overline{\Omega})$; see Theorem 8.2 of \cite{Trudinger}. 

% Linfty bounds
\par Let us now proceed to derive some a priori bounds on solutions of \eqref{penalized}. First, we note that the comparison principle associated with \eqref{penalized} implies  
\begin{equation}\label{uepsLinf}
\underline{u}(x)\le u^\epsilon(x)\le \overline{u}(x)\quad x\in \overline{\Omega}. 
\end{equation}
Moreover, according to the Alexandroff-Bakelman-Pucci maximum principle (Theorem 3.6 in \cite{CC}), there is a constant $C=C(\diam(\Omega),n,\lambda,\Lambda)$ such that
$$
\overline{u}(x)\le C\left(|\varphi|_{L^\I(\partial \Omega)} +|f|_{L^\infty(\Omega)} + |F(O_n,\cdot)|_{L^\infty(\Omega)}\right), \quad x\in \Omega.
$$ 
Here $O_n\in\Sn$ is the zero matrix.  As $\underline{u}|_{\partial\Omega}=\varphi$, we may combine the upper bound for $u^\epsilon$ with \eqref{uepsLinf} to get 
$$
|u^\epsilon|_{L^\infty(\Omega)}\le C\left( |\underline{u}|_{L^\I(\Omega)} +|f|_{L^\infty(\Omega)} + |F(O_n,\cdot)|_{L^\infty(\Omega)}\right).
$$
\par Using this $L^\infty$ estimate, we will bound $|Du^\epsilon(x)|$ independently of  $\epsilon\in (0,1)$ for $x$ belonging to compact subsets of $\Omega$. To this end, we 
shall make use of the uniform convexity assumption on $H$ \eqref{Hassump}, which in turn implies
\begin{align}\label{Coercive}
\begin{cases}
H(p)\ge H(0)+DH(0)\cdot p+\frac{\theta}{2}|p|^2\\
DH(p)\cdot p-H(p)\ge -H(0)+\frac{\theta}{2}|p|^2\\
|DH(p)|\le |DH(0)| + \sqrt{n}\Theta|p|
\end{cases}(p\in \R^n).
\end{align}
\par In our computations below there will be several constants that will depend on the various ``data" associated with the boundary value problem \eqref{penalized} but they will all be {\it independent} of all 
$\epsilon$ sufficiently small. It will be important for us to keep track of the dependence of constants on the data in order to later remove the smoothness assumption \eqref{ExtraSmooth}. For convenience, 
we make the following list of parameters 
\begin{align}\label{listOfPar}
\Pi:=(\lambda,\Lambda,\theta,\Theta, \Upsilon,n,\diam(\Omega), H(0),|DH(0)|,|F(O_n,\cdot)|_{L^\I(\Omega)}, |f|_{W^{1,\infty}(\Omega)}|, |\underline{u}|_{W^{1,\I}(\Omega)})
\end{align}
that we shall quote several times below.  

\begin{lem}\label{GradBoundUeps}
Let $\Sigma\subset\subset \Omega$ be open. There is a constant $C$ such that 
$$
|Du^\epsilon(x)|\le C, \quad x\in \Sigma
$$
for $\epsilon \in (0,1)$.   Here $C$ depends on $\Pi$ and $1/\dist(\Sigma,\partial \Omega)$.
\end{lem}
\begin{proof}
1. Let $\eta\in C^\I_c(\Omega)$ satisfy $0\le \eta\le 1$ and set
$$
M_\epsilon:=\sup_{x\in \Omega}|\eta(x)Du^\epsilon(x)|.
$$
In order to prove the lemma, we will first bound $M_\epsilon$ uniformly in $\epsilon\in (0,1)$. To this end, we study the maximum values of the function
$$
v^\epsilon(x):=\frac{1}{2}\eta^2(x)|Du^\epsilon(x)|^2 -  \alpha_\epsilon (u^\epsilon(x)-\underline{u}(x)) 
$$
for a nonnegative constant $\alpha_\epsilon$ to be selected below. We emphasize that in the computations that follow, all constants will depend only on $|\eta|_{W^{2,\infty}(\Omega)}$ and the list $\Pi$.

% Deriving the Linearized equation for Bernstein function
\par 2. Differentiating the PDE \eqref{penalized} with respect to $x_k$ gives 
\begin{equation}\label{1stDerPenEqn}
F_{M_{ij}}(D^2u^\epsilon,x)u^\epsilon_{x_i x_j x_k} + F_{x_k}(D^2u^\epsilon,x)+\beta'_\epsilon(H(Du^\epsilon))DH(Du^\epsilon)\cdot Du^\epsilon_{x_k}=f_{x_k}. 
\end{equation}
Direct computation combined with \eqref{1stDerPenEqn} also yields
\begin{align}\label{BernIdentity1}
F_{M_{ij}}v_{x_i x_j}+ \beta' H_{p_k}v_{x_k}&=\left(F_{M_{ij}}\eta_{x_i}\eta_{x_j} + \eta F_{M_{ij}}\eta_{x_i x_j}\right)|Du|^2  + \nonumber \\
& \quad\quad 4F_{M_{ij}}\eta\eta_{x_i} Du\cdot Du_{x_j}  +\eta^2 F_{M_{ij}}Du_{x_i}\cdot Du_{x_j}\nonumber \\
&\quad\quad - \beta' H_{p_k}(\alpha (u_{x_k}-\underline{u}_{x_k})- \eta\eta_{x_k}|Du|^2) -\eta^2u_{x_k}F_{x_k}\nonumber \\
& \quad\quad +\eta^2 u_{x_k}f_{x_k} - \alpha F_{M_{ij}}(u_{x_i x_j}-\underline{u}_{x_i x_j}).
\end{align}
Moreover, the convexity of $M\mapsto F(M,x)$, the various properties of
$\beta=\beta_\epsilon$, and \eqref{ubarEqn} leads to
% Convexity argument
\begin{align*}
-F_{M_{ij}}(u_{x_i x_j}-\underline{u}_{x_i x_j})&:=-F_{M_{ij}}(D^2u,x)(D^2u-D^2\underline{u})_{ij} \\ 
& \le F(D^2\underline{u},x) - F(D^2u,x)\\
&= F(D^2\underline{u},x)-f(x) +\beta(H(Du))\\
&< \beta(H(Du))\\
&\le \beta(H(D\underline{u})) +\beta'(H(Du))(H(Du) - H(D\underline{u}))\\
&=\beta'(H(Du))(H(Du) - H(D\underline{u})).
\end{align*}
And substituting this inequality into \eqref{BernIdentity1} gives 
\begin{align}\label{BernIdentity11}
F_{M_{ij}}v_{x_i x_j}+ \beta' H_{p_k}v_{x_k}&\leq \left(F_{M_{ij}}\eta_{x_i}\eta_{x_j} + \eta F_{M_{ij}}\eta_{x_i x_j}\right)|Du|^2  + \nonumber \\
& \quad\quad 4F_{M_{ij}}\eta\eta_{x_i} Du\cdot Du_{x_j}  +\eta^2 F_{M_{ij}}Du_{x_i}\cdot Du_{x_j}\nonumber \\
&\quad\quad - \beta'\left( \alpha(H_{p_k}(u_{x_k}-\underline{u}_{x_k})- H+H(D\underline{u}))- \eta H_{p_k}\eta_{k}|Du|^2\right) \nonumber \\
& \quad\quad +\eta^2 u_{x_k}f_{x_k}-\eta^2u_{x_k}F_{x_k}.
\end{align}
\par 3. Now let $x_0\in \overline{\Omega}$ be a point maximizing $v$. For a given $\alpha>0$, if $x_0\in \partial \Omega$
\begin{equation}\label{v1Upp}
v\le  C (\alpha+1).
\end{equation}
Otherwise $x_0\in \Omega$. In this case, if $\beta'\le 1<\frac{1}{\epsilon}$ then $H=H(Du(x_0))\le 2\epsilon\le 2$ by \eqref{betaAss}. In view of \eqref{Coercive}, $|Du(x_0)|^2$ is then bounded above uniformly in $\epsilon\in (0,1)$ which gives again \eqref{v1Upp}.

\par Let us now study the case $\beta'\ge 1$. We will use that $(F_{M_{ij}})$ satisfies \eqref{barFXassump2}. In particular, 
$$
\eta^2 F_{M_{ij}}Du_{x_i}\cdot Du_{x_j}\le -\eta^2\lambda |D^2u|^2
$$
which controls the term $4F_{M_{ij}}\eta\eta_{x_i} Du\cdot Du_{x_j}$ up to an expression of the form $C(|Du|^2+1)$.
This observation combined with the necessary conditions $Dv(x_0)=0$ and $D^2v(x_0)\le 0$ allow us to evaluate  \eqref{BernIdentity11} at the point $x_0$ and arrive
at 
\begin{align*}
0 & \le C(|Du|^2+1) - \beta'\left( \alpha(H_{p_k}(u_{x_k}-\underline{u}_{x_k})- H+H(D\underline{u}) )- \eta H_{p_k}\eta_{x_k}|Du|^2\right)\\
& \le C(|Du|^2+1) - \beta'\left(\frac{1}{2}\alpha\theta|Du-D\underline{u}|^2- \eta H_{p_k}\eta_{x_k}|Du|^2\right)\\
&\le C(|Du|^2+1)\beta' -\beta'\left( \frac{1}{4}\alpha\theta\left(|Du|^2 - C\right)- \eta C_0(1+|Du|)|Du|^2\right)\\
&\le \beta'\left[C(|Du|^2+1) -\left(  \frac{1}{4}\alpha\theta\left(|Du|^2 - C\right)- \eta C_0(1+|Du|)|Du|^2\right)\right].
\end{align*}
Multiplying through by $4\eta(x_0)^2$ gives
\begin{equation}\label{BernIdentity111}
0\le \beta'\left[C(\eta^2|Du|^2+1) -\alpha\theta\left( \eta^2|Du|^2 - C\right) + 4C_0(\eta^3|Du|^3+\eta^2|Du|^2)\right].
\end{equation}
\par 4. Now select 
$$
\alpha:=\frac{5C_0}{\theta} M_\epsilon
$$
and note $\alpha \ge (5C_0/\theta)|\eta(x_0)Du(x_0)|$. In particular, \eqref{BernIdentity111} becomes 
$$
0\le \beta'\left[C(\eta^2|Du|^2+1) -5C_0\eta|Du|\left(\eta^2|Du|^2 - C\right) + 4C_0(\eta^3|Du|^3+\eta^2|Du|^2)\right].
$$
As $\beta'\ge 1$, the expression in the brackets must be nonnegative. It follows that $\eta(x_0)^3|Du(x_0)|^3$ is bounded above 
by a quadratic function of $\eta(x_0)|Du(x_0)|$. Hence, $\eta(x_0)|Du(x_0)|$ is bounded uniformly in $\epsilon\in (0,1)$.  Again we are able to conclude \eqref{v1Upp}.
 Therefore, with our choice of $\alpha=\alpha_\epsilon$, in all cases we have 
$$
M_\epsilon^2=\sup_{\Omega}(|\eta Du^\epsilon|^2 )=2\sup_{\Omega}(v^\epsilon +\alpha_\epsilon (u^\epsilon-\underline{u} ))\le C(\alpha_\epsilon+1)= C\left(\frac{5C_0}{\theta}M_\epsilon+1\right).
$$
As a result, $M_\epsilon$ is bounded uniformly in $\epsilon\in (0,1)$. 

\par 5. Now assume $\Sigma\subset\subset \Omega$ is open and $r<\frac{1}{2}\dist(\Sigma,\partial\Omega)$.  Also let $y\in\Sigma$ and note by our choice in $r$, $B_{2r}(y)\subset \Omega$.  Next, choose an $\eta\in C^\infty_c(B_{2r}(y))$ with $0\le\eta\le 1$, $\eta\equiv 1$ in $B_r(y)$, and 
$$
|D\eta(x)|\le\frac{C}{r} \quad \text{and}\quad |D^2\eta(x)|\le\frac{C}{r^2}
$$
for $x\in B_{2r}(y)$. From the argument above 
$$
|Du^\epsilon(x)|\le C, \quad x\in B_r(y)
$$
for a constant $C$ depending on $1/r$ and the list $\Pi$. As $y\in \Sigma$ is arbitrary, the assertion follows. 
 \end{proof}
% W2p  Estimates 
We now pursue $W^{2,p}_\text{loc}(\Omega)$ estimates on solutions. Our main assertion is that the function $x\mapsto \beta_\epsilon(H(Du^\epsilon(x)))$ is locally bounded, independently of all sufficiently small 
$\epsilon$. 
\begin{lem}\label{betaBound}
Let $\Sigma\subset\subset\Omega$. There is a constant $C$ such that 
$$
0\le \beta_\epsilon(H(Du^\epsilon(x)))\le C, \quad x\in \Sigma
$$
for all $\epsilon \in (0,1)$. Here $C$ depends on the list $\Pi$ and $1/\dist(\Sigma,\partial \Omega)$.
\end{lem}
From equation \eqref{penalized},
$$
F(D^2u^\epsilon,x)=-\beta_\epsilon(H(Du^\epsilon(x)))+f(x).
$$
And in view of Lemma \ref{GradBoundUeps} and Lemma \ref{betaBound}, the right hand side above is bounded in $L^\infty_{\text{loc}}(\Omega)$ independently of $\epsilon\in (0,1)$.  Therefore, we can apply the interior $W^{2,p}$ estimate for fully nonlinear elliptic equations due to L. Caffarelli (Theorem 1 of \cite{CaffAnn}) to obtain a $W^{2,p}_{\text{loc}}(\Omega)$ estimate estimate on $u^\epsilon$ that is independent of $\epsilon\in (0,1)$. 
\begin{cor}\label{CorW2p}
Let $p\ge 1$. For a given $\Sigma\subset\subset\Omega$, there is a constant $C$ for which
$$
|D^2u^\epsilon|_{L^p(\Sigma)}\le C
$$
for all $\epsilon \in (0,1)$. Here $C$ depends $p$, $\Sigma$, $1/\dist(\Sigma,\partial \Omega)$ and the list $\Pi$. 
\end{cor}
\begin{proof} It suffices to verify the assertion for each $p>n$. By Theorem 7.1 of \cite{CC}, there is a constant $\beta_0$ such that for each
for $B_{2r}(x_0)\subset \Omega$
\begin{equation}\label{CaffW2pEst}
\left(\gfint\; |D^2u^\epsilon(x)|^pdx\right)^{1/p}\le \frac{c_0}{r^2}\left\{|u^\epsilon|_{L^\I(B_{2r}(x_0))} +|-\beta_\epsilon(H(Du^\epsilon))+f -F(O_n,\cdot) |_{L^\infty(B_{2r}(x_0))}\right\}
\end{equation}
as long as 
$$
\sup_{M\neq 0}\frac{|F(M,x)-F(M,y)|}{|M|}\le \beta_0, \quad x,y\in B_{2r}(x_0).
$$
Here $\beta_0$ and $c_0$ only depend on $\lambda,\Lambda,p$ and $n$.  A careful inspection of the proof of Theorem 7.1 in \cite{CC} (and Lemma 7.9) shows an estimate of the form \eqref{CaffW2pEst} holds provided 
\begin{equation}\label{beta1cond}
\sup_{M\neq 0}\frac{|F(M,x)-F(M,y)|}{|M|+1}\le \beta_1, \quad x,y\in B_{2r}(x_0)
\end{equation}
for a constant $\beta_1$ depending only on $\lambda, \Lambda, p$, and $n$.   See also the beginning of the proof of Theorem 8.1 in \cite{CC} for a related condition. 

\par In view of the assumption \eqref{FXassump}, it follows that \eqref{beta1cond} holds for each  $B_{2r}(x_0)\subset\Omega$ with $r\le \frac{\beta_1}{2\Upsilon}.$  Now fix 
$$
0<r<\min\left\{\frac{1}{4}\dist(\Sigma,\partial\Omega),\frac{\beta_1}{2\Upsilon}\right\}.
$$
With this choice of $r$, the estimate \eqref{CaffW2pEst} combined with the proofs of Lemma \ref{GradBoundUeps} and Lemma \ref{betaBound} imply
\begin{equation}\label{CaffW2pEst2}
|D^2u^\epsilon|_{L^p(B_r(x_0))}\le C_1
\end{equation}
for some constant $C_1$ depending on  $p$, $r$, and the list $\Pi$. By the compactness of $\overline{\Sigma}$, there are finitely many balls $B_r(x_1),\dots, B_r(x_m)$ that cover $\overline{\Sigma}$ with 
$\{x_i\}_{i=1,\dots,m}\subset \Sigma$. Thus, 
\begin{align*}
\int_{\Sigma}|D^2u^\epsilon(x)|^pdx & \le\int_{\cup^m_{i=1}B_r(x_i)}|D^2u^\epsilon(x)|^pdx \\
& \le \sum^{m}_{i=1}\int_{B_r(x_i)}|D^2u^\epsilon(x)|^pdx \\
&\le mC_1^p.
\end{align*}
\end{proof}
% Elementary Matrix Ineq
In the proof of Lemma \ref{betaBound} below, we will make use of the following  elementary inequality. If $S, T\in \Sn$ and 
the eigenvalues of $S$ and $T$ are greater than or equal to $s\ge 0$ and $t\ge 0$, respectively, then
\begin{equation}\label{ElemIneq}
S\cdot (X T X)\ge s t |X|^2, \quad X\in \Sn. 
\end{equation}
To verify inequality \eqref{ElemIneq}, we first write $S=O\text{diag}(s_1,\dots s_n)O^t$ for some $s_1,\dots, s_n\ge s$ and orthogonal $n\times n$ matrix $O$. Then we calculate
\begin{align*}
S\cdot (X T X) & =\text{tr}(SXTX)\\
&=\text{tr}(O\text{diag}(s_1,\dots s_n)O^tXTX)\\
&=\text{tr}(\text{diag}(s_1,\dots s_n)O^tXTXO)\\
&=\sum^n_{i=1}s_i (O^tXTXO)_{ii}\\
&=\sum^n_{i=1}s_i (O^tXTXO)e_i\cdot e_i\\
&=\sum^n_{i=1}s_i \left(TXOe_i\cdot XOe_i\right)\\
&\ge s\sum^n_{i=1}TXOe_i\cdot XOe_i\\
&\ge st\sum^n_{i=1}XOe_i\cdot XOe_i\\
&=st|X|^2.
\end{align*}
\begin{proof} (of Lemma \ref{betaBound})
1. Let $\eta\in C^\infty_c(\Omega)$ satisfy $0\le \eta\le 1$ and set
$$
v^\epsilon(x)=\eta(x)^2\beta_\epsilon(H(Du^\epsilon(x))), \quad x\in \Omega.
$$
We first attempt to bound $v^\epsilon$ from above by a universal constant.  Again, we suppress $\epsilon$ dependence and function arguments to compute 
\begin{align}\label{W2pCero}
F_{M_{ij}}v_{x_i x_j} +\beta'H_{p_k}v_{x_k} & = 2(F_{M_{ij}}\eta_{x_i}\eta_{x_j} + \eta F_{M_{ij}}\eta_{x_i x_j} )\beta + 4 F_{M_{ij}}\eta\eta_{x_i}\beta' DH\cdot Du_{x_j} \nonumber \\
& \quad + \eta^2\beta'' F_{M_{ij}}(DH\cdot Du_{x_i})(DH\cdot Du_{x_j}) +\eta^2\beta' F_{M_{ij}}D^2HDu_{x_i}Du_{x_j}  \nonumber \\
& \quad + \beta'H_{p_k}\left[\eta^2\left(f_{x_k}- F_{x_k}\right) +2\eta\eta_{x_k}\beta\right].
\end{align}
Let us now estimate each term in the identity \eqref{W2pCero}.  Below, each constant $C$ will depend only on the list $\Pi$ and $|\eta|_{W^{2,\infty}(\Omega)}$.
\par 2. By the convexity of $\beta=\beta(z)$ and the previous lemma, 
$$
\eta \beta(H(Du))\le  \eta H(Du)\beta'(H(Du))\le C\beta'(H(Du))
$$ 
for some constant $C$. Thus
\begin{equation}\label{W2pUno}
(F_{M_{ij}}\eta_{x_i}\eta_{x_j} + \eta F_{M_{ij}}\eta_{x_i x_j} )\beta\le C\beta'.
\end{equation}
Likewise 
\begin{equation}\label{W2pDos}
4 F_{M_{ij}}\eta\eta_{x_i}\beta' DH\cdot Du_{x_j} \le C\beta'\eta |D^2u|;
\end{equation}
and by the convexity of $\beta$ and the ellipticity of $F$
\begin{equation}\label{W2pDosDos}
\quad \eta^2\beta'' F_{M_{ij}}(DH\cdot Du_{x_i})(DH\cdot Du_{x_j}) \le 0.
\end{equation}
Moreover, appealing to \eqref{ElemIneq}
\begin{equation}\label{W2pTres}
\eta^2\beta' F_{M_{ij}}D^2HDu_{x_i}Du_{x_j}\le -\beta'\theta\lambda \eta^2|D^2u|^2.
\end{equation}
\par 3. By assumptions \eqref{FUnifEll} and \eqref{FXassump}, $|F(M,x)|\le C(|M|+1)$. Thus,
$$
0\le \beta = f -F\le C(1+|D^2u|).
$$
It now follows from \eqref{barFXassump2} that
\begin{align}\label{W2pQuarter}
 \beta'H_{p_k}\left[\eta^2\left(f_{x_k}- F_{x_k}\right) +2\eta\eta_{x_k}\beta\right]
 &\le C\beta' |DH|\left\{\eta^2(1+|D^2u|) + \eta |D\eta|(1+|D^2u|)\right\}\nonumber \\
 &\le C\beta' \eta(1+|D^2u|)(\eta +|D\eta|)\nonumber \\
 &\le C\beta'(1+\eta |D^2u|).
\end{align}
Combining \eqref{W2pUno}, \eqref{W2pDos} \eqref{W2pDosDos}, \eqref{W2pTres}, and \eqref{W2pQuarter}, \eqref{W2pCero} becomes
\begin{equation}\label{W2pCinco}
F_{M_{ij}}v_{x_i x_j} +\beta'H_{p_k}v_{x_k} \le \beta'\left(C +C\eta|D^2u|-\theta\lambda \eta^2|D^2u|^2\right).
\end{equation}
\par 4. Now suppose that $v(x_0)=\max_{\overline{\Omega}} v$ for some $x_0\in \Omega$.  Then by calculus 
$$
Dv(x_0)=0\quad \text{and}\quad D^2v(x_0)\le 0.
$$
Appealing to \eqref{W2pCinco}, we have at $x_0$
$$
0\le \beta'\left(C +C\eta|D^2u|-\theta\lambda \eta^2|D^2u|^2\right).
$$
If $\beta'=0$, then $\beta=0$; so we may as well assume $\beta'>0$. In this case, at the point $x_0$
$$
\theta\lambda \eta^2|D^2u|^2\le C +C\eta|D^2u|,
$$
which in turn implies $\eta(x_0)|D^2u(x_0)|\le C$. As a result, 
$$
v\le v(x_0)\le \eta(x_0)\beta(H(Du(x_0)))\le C(1+\eta(x_0) |D^2u(x_0)|)\le C.
$$
As in the proof of Lemma \ref{GradBoundUeps}, we conclude by choosing appropriate test functions $\eta$ to localize our uniform supremum bound on $v^\epsilon$. 
\end{proof}
% W2infinity estimate  
Now we turn to establishing a priori $W^{2,\infty}_\text{loc}(\Omega)$ estimates for solutions of \eqref{penalized} that are independent of $\epsilon\in (0,1)$.  We 
now make the specific assumption that $F$ is independent of $x$. That is, we assume $F(M,x)=F(M)$.  The function $u^\epsilon$ now satisfies the penalized equation 
\begin{equation}\label{penalizedNox}
F(D^2u^\epsilon)+\beta_\epsilon(H(Du^\epsilon))=f(x), \quad x\in \Omega.
\end{equation}
Differentiating \eqref{penalizedNox} twice with respect any direction $\xi\in \R^n$ ($|\xi|=1$) gives
\begin{align}\label{DiffTwiceEqn}
F_{M_{ij}M_{i'j'}}(D^2u^\epsilon)u^\epsilon_{x_i x_j \xi}u^\epsilon_{x_{i' }x_{j'} \xi}+ F_{M_{ij}}(D^2u^\epsilon)u^\epsilon_{x_i x_j \xi\xi} +\beta_\epsilon''(H(Du^\epsilon))(DH(Du^\epsilon)\cdot Du^\epsilon_\xi)^2 \nonumber \\ 
+\beta_\epsilon'(H(Du^\epsilon))(D^2H(Du^\epsilon)Du^\epsilon_\xi\cdot Du^\epsilon_\xi + DH(Du^\epsilon)\cdot Du_{\xi\xi} )=f_{\xi\xi}.
\end{align}
\par We also remark that in obtaining a $W^{2,\infty}_\text{loc}(\Omega)$ estimate it is sufficient to obtain a one sided bound
\begin{equation}\label{SemiConcaveEnough}
D^2u^\epsilon(x)\le CI_n\quad x\in \Sigma
\end{equation}
for each open $\Sigma\subset\subset\Omega$.  Indeed, the uniform ellipticity of $F$ would then imply
$$
\lambda\tr(CI_n -D^2u^\epsilon)\le F(D^2u^\epsilon)-F(D^2u^\epsilon+(CI_n-D^2u^\epsilon))\le \Lambda\tr(CI_n -D^2u^\epsilon).
$$
Moreover, since $F(D^2u^\epsilon)=f-\beta_\epsilon(DH(Du^\epsilon))$ and  $F(D^2u^\epsilon+(CI_n-D^2u^\epsilon))=F(CI_n)$ are locally bounded, independently of $\epsilon \in (0,1)$, $\tr(CI_n-D^2u^\epsilon)$ would enjoy a similar bound. It would then follow that $-\Delta u^\epsilon$ is bounded on $\Sigma$ independently of $\epsilon\in (0,1)$. Let us write $|\Delta u^\epsilon|_{L^\infty(\Sigma)}\le C_1,$ in this case. 

\par Suppose $|\xi|=1$ and $\xi,\xi_1,\xi_2,\dots, \xi_{n-1}$ is an orthonormal basis for $\R^n$. From the upper bound \eqref{SemiConcaveEnough} and the bound on $\Delta u^\epsilon$, we would have for $x\in \Sigma$
$$
u^\epsilon_{\xi\xi}(x)= \Delta u^\epsilon(x)-\sum^{n-1}_{i=1}u^\epsilon_{\xi_i\xi_i}(x)\ge -C_1- (n-1)C.
$$
It would then follow
$$
D^2u^\epsilon(x)\ge - (C_1+(n-1)C)I_n,\quad x\in \Sigma.
$$
Therefore, in proving the assertion below, we just need to bound $D^2u^\epsilon$ from above. 

\begin{lem}\label{W2INflem}
For each $\Sigma\subset\subset\Omega$, there is a constant $C$ for which
$$
|D^2u^\epsilon|_{L^\I(\Sigma)}\le C
$$
for all $\epsilon \in (0,1)$. Here $C$ depends on the list $\Pi$, $|f|_{W^{2,\I}(\Omega)}$ and $1/\dist(\Sigma,\partial \Omega)$.
\end{lem}
\begin{proof}
We first bound the function
$$
v^\epsilon(x):=\eta^2(x)\left\{\left[(u^\epsilon_{\xi\xi}(x))^+\right]^2 +\alpha\beta_\epsilon(H(Du^\epsilon(x))) + \mu|Du^\epsilon(x)|^2\right\}, \quad x\in \Omega
$$
from above, independently of $\epsilon \in (0,1)$. Here $\xi\in \R^n$ with $|\xi|=1$, $\eta\in C^\infty_c(\Omega)$ satisfies $0\le \eta\le 1$ and $\alpha$ and $\mu$ are positive numbers
that we will choose below.  

\par At any point $x$ for which $u_{\xi\xi}>0$ we suppress function arguments, $\epsilon$ dependence, and use \eqref{1stDerPenEqn} and \eqref{DiffTwiceEqn} to compute 
\begin{align}\label{W2INF}
F_{M_{ij}}v_{x_i x_j} +\beta'H_{p_k}v_{x_k} & = 2(F_{M_{ij}}\eta_{x_i}\eta_{x_j} + \eta F_{M_{ij}}\eta_{x_i x_j} )( (u_{\xi\xi})^2 +\alpha \beta + \mu|Du|^2) \nonumber \\
& \quad +4 F_{M_{ij}}\eta\eta_{x_i} ( 2 u_{\xi\xi} u_{\xi\xi x_j} +\alpha \beta'DH\cdot Du_{x_j} + 2\mu Du\cdot Du_{x_j}) \nonumber  \\
& \quad + \eta^2\left[ 2F_{M_{ij}}u_{\xi\xi x_i}u_{\xi\xi x_j}  +2 u_{\xi\xi}\left(- F_{M_{ij}M_{i'j'}}u_{x_i x_j\xi}u_{x_i'x_j'\xi}   \right.  \right. \nonumber \\
& \quad -\beta''(DH\cdot Du_\xi)^2 - \beta' D^2HDu_\xi\cdot Du_\xi  +f_{\xi\xi}\Big) \nonumber \\
& \quad + \alpha \beta''F_{M_{ij}}(DH\cdot Du_{x_i})(DH\cdot Du_{x_j})  \nonumber \\
& \quad + \alpha \beta'(F_{M_{ij}}D^2HDu_{x_i}\cdot Du_{x_j}+ H_{p_k}f_{x_k})\nonumber \\
& \quad + 2\mu(F_{M_{ij}}Du_{x_i}\cdot Du_{x_j}+ u_{x_k}f_{x_k}) \Big] \nonumber \\
& \quad + 2\beta'\eta\eta_{x_k} H_{p_k}( (u_{\xi\xi})^2 +\alpha \beta + \mu|Du|^2).
\end{align}
\par Let us now estimate each term on the right hand side of \eqref{W2INF} individually; we will use the conclusions of the previous lemmas, assumed positivity of $u_{\xi\xi}>0$, the convexity of $F$ and $H$, and the properties of $\beta$ \eqref{betaAss}.  Each constant below may depend on $\Pi$, $|f|_{W^{2,\I}(\Omega)}$ and $|\eta|_{W^{2,\infty}(\Omega)}$ but will be independent of $\alpha$ and $\mu$.  
%\allowdisplaybreaks
\begin{align}\label{W2INFindep}
\begin{cases}
2(F_{M_{ij}}\eta_{x_i}\eta_{x_j} + \eta F_{M_{ij}}\eta_{x_i x_j} )( (u_{\xi\xi})^2 +\alpha \beta + \mu|Du|^2)\le C(|D^2u|^2 + \alpha +\mu) \\ \\
8 F_{M_{ij}}\eta\eta_{x_i} u_{\xi\xi} u_{\xi\xi x_j} \le \lambda\eta^2|Du_{\xi\xi}|^2 +C|D^2u|^2 \\\\
4 F_{M_{ij}}\eta\eta_{x_i} (\alpha \beta'DH\cdot Du_{x_j} + 2\mu Du\cdot Du_{x_j})\le C(\alpha\beta'|D^2u| +\mu|D^2u|) \\\\
 2\eta^2F_{M_{ij}}u_{\xi\xi x_i}u_{\xi\xi x_j}\le -2\lambda \eta^2|Du_{\xi\xi}|^2  \\\\
 2 u_{\xi\xi}\left(- F_{M_{ij}M_{i'j'}}u_{x_i x_j\xi}u_{x_i'x_j'\xi}-\beta''(DH\cdot Du_\xi)^2 - \beta' D^2HDu_\xi\cdot Du_\xi   \right)\le 0 \\\\
 2 \eta^2u_{\xi\xi}f_{\xi\xi}\le C|D^2u|\\\\
 \alpha \beta''F_{M_{ij}}(DH\cdot Du_{x_i})(DH\cdot Du_{x_j}) \le 0\\\\
 \alpha \beta'\eta^2(F_{M_{ij}}D^2HDu_{x_i}\cdot Du_{x_j}+ H_{p_k}f_{x_k})\le \alpha\beta'(-\lambda\theta|D^2u|^2+C )\\\\
  2\mu\eta^2(F_{M_{ij}}Du_{x_i}\cdot Du_{x_j}+ u_{x_k}f_{x_k})\le 2\mu(-\lambda |D^2u|^2+C) \\\\
  2\beta'\eta\eta_{x_k} H_{p_k}( (u_{\xi\xi})^2 +\alpha \beta + \mu|Du|^2)\le 2\beta'C(|D^2u|^2+\alpha+\mu)
\end{cases}.
\end{align}
Combining \eqref{W2INF} with \eqref{W2INFindep} gives 
\begin{align}\label{W2INFnew}
F_{M_{ij}}v_{x_i x_j} +\beta'H_{p_k}v_{x_k}& \le \beta'\left\{C_0\left[\alpha(|D^2u|+1)+\mu +|D^2u|^2\right] -\alpha\theta\lambda |D^2u|^2 \right\}\nonumber \\
& \quad + \left\{C_0\left[\mu(|D^2u|+1)+\alpha +1+|D^2u|^2\right] -2\mu\lambda |D^2u|^2 \right\}
\end{align}
for some universal constant $C_0>0$. Recall this inequality holds on $\Omega\cap\{u_{\xi\xi}>0\}.$ 

\par Select $x_0\in \overline{\Omega}$ that maximizes $v$.  If $x_0\in \partial\Omega$ or $u_{\xi\xi}(x_0)\le 0$, then the desired upper bounds on $v$ are immediate.  Alternatively, 
if $x_0\in\Omega\cap\{u_{\xi\xi}>0\}$, we have from \eqref{W2INFnew} that at $x_0$
\begin{align}\label{W2INFnewATxo}
0& \le \beta'\left\{C_0\left[\alpha(|D^2u|+1)+\mu +|D^2u|^2\right] -\alpha\theta\lambda |D^2u|^2 \right\}\nonumber \\
& \quad + \left\{C_0\left[\mu(|D^2u|+1)+\alpha +1+|D^2u|^2\right] -2\mu\lambda |D^2u|^2 \right\}.
\end{align}
One of the two parenthesized expressions in \eqref{W2INFnewATxo} must be nonnegative, or the inequality in \eqref{W2INFnewATxo} cannot hold. We now choose
\begin{equation}\label{ChoiceAM}
\alpha:=\frac{2C_0}{\theta\lambda}\quad\text{and}\quad \mu:=\frac{C_0}{\lambda}. \nonumber
\end{equation}
In view of the expressions in the parentheses in inequality \eqref{W2INFnewATxo}, it now follows that $|D^2u(x_0)|$ and in turn $v(x_0)$ is bounded above independently of $\epsilon\in (0,1)$. As explained in the proof of Lemma \ref{GradBoundUeps}, 
we conclude after choosing appropriate test functions $\eta$ to localize our uniform $L^\infty$ bound on $v^\epsilon$.
\end{proof}
% Remark on why x-dependence is hard 
\begin{rem}\label{WiegRem}
If $F$ depends explicitly on $x$, \eqref{W2INF} becomes
\begin{align*}\label{W2INFwX}
F_{M_{ij}}v_{x_i x_j} +\beta'H_{p_k}v_{x_k} & = 2(F_{M_{ij}}\eta_{x_i}\eta_{x_j} + \eta F_{M_{ij}}\eta_{x_i x_j} )( (u_{\xi\xi})^2 +\alpha \beta + \mu|Du|^2) \nonumber \\
& \quad +4 F_{M_{ij}}\eta\eta_{x_i} ( 2 u_{\xi\xi} u_{\xi\xi x_j} +\alpha \beta'DH\cdot Du_{x_j} + 2\mu Du\cdot Du_{x_j}) \nonumber  \\
& \quad + \eta^2\left[ 2F_{M_{ij}}u_{\xi\xi x_i}u_{\xi\xi x_j}  +2 u_{\xi\xi}\left(- F_{M_{ij}M_{i'j'}}u_{x_i x_j\xi}u_{x_i'x_j'\xi}   \right.  \right. \nonumber \\
& \quad - 2F_{M_{ij}, \xi}u_{x_ix_j \xi}- F_{\xi\xi}-\beta''(DH\cdot Du_\xi)^2 - \beta' D^2HDu_\xi\cdot Du_\xi  +f_{\xi\xi}\Big) \nonumber \\
& \quad + \alpha \beta''F_{M_{ij}}(DH\cdot Du_{x_i})(DH\cdot Du_{x_j})  \nonumber \\
& \quad + \alpha \beta'(F_{M_{ij}}D^2HDu_{x_i}\cdot Du_{x_j}+ H_{p_k}(f_{x_k}-F_{x_k}))\nonumber \\
& \quad + 2\mu(F_{M_{ij}}Du_{x_i}\cdot Du_{x_j}+ u_{x_k}(f_{x_k}-F_{x_k})) \Big] \nonumber \\
& \quad + 2\beta'\eta\eta_{x_k} H_{p_k}( (u_{\xi\xi})^2 +\alpha \beta + \mu|Du|^2).
\end{align*}
Unfortunately, the third derivative term $2F_{M_{ij}, \xi}u_{x_ix_j \xi}$ doesn't have a sign nor is it obviously controlled by $2F_{M_{ij}}u_{\xi\xi x_i}u_{\xi\xi x_j}$. This 
is the same issue L. C. Evans discovered in the case $F(M,x)=-a(x)\cdot M$. M. Wiegner circumvented this issue by using the more general Bernstein function 
$$
v^\epsilon(x):=\eta^2(x)\left\{|D^2u^\epsilon(x)|^2 +\alpha\beta_\epsilon(H(Du^\epsilon(x))) + \mu|Du^\epsilon(x)|^2\right\}, \quad x\in \Omega.
$$
However, when one carries out the computation of $F_{M_{ij}}v_{x_i x_j} +\beta'H_{p_k}v_{x_k}$ one encounters the term
\begin{equation}\label{GenHEssionFterm}
-u^\epsilon_{x_k x_l}F_{M_{ij}M_{i'j'}}u^\epsilon_{x_kx_i x_j}u^\epsilon_{x_l x_{i'}x_{j'}} 
\end{equation}
and it is unclear how to exploit the convexity of $F$. Of course when $F(M,x)=-a(x)\cdot M$, \eqref{GenHEssionFterm} vanishes and this
is a nonissue. 
\end{rem}

%%%%%%%%%%%%%%%%%%%%%%%%%%%%%%%%%% Convergence %%%%%%%%%%%%%%%%%%%%%%%%%%%%%%%%%
\section{Convergence}\label{SecConv}
%  viscosity solution property in limit 
In this section, we prove parts $(ii)$ and $(iii)$ of Theorem \ref{mainTHM}.  We first 
prove these assertions under the smoothness assumption \eqref{ExtraSmooth}, then remove these assumptions 
by an approximation argument. Moreover, we provide all the details for part $(ii)$ and omit the proof of part $(iii)$ as it follows similarly to part $(ii)$. The main insight for part $(iii)$ was accomplished in our a priori $W^{2,\infty}_\text{loc}(\Omega)$ estimate of $u^\epsilon$ in  Lemma \ref{W2INflem}. 

\begin{proof} (part $(ii)$ of Theorem \ref{mainTHM})
1. In Lemma \ref{GradBoundUeps} and Corollary \ref{CorW2p}, we established the existence of a unique classical solution $u^\epsilon\in C^\I(\Omega)\cap C(\overline{\Omega})$ 
of \eqref{penalized} that satisfies 
$$
\begin{cases}
|u^\epsilon|_{L^\I(\Omega)}\le C_1(\Pi)\\
|Du^\epsilon|_{L^\I(\Sigma)}\le C_2(1/\dist(\Sigma,\partial\Omega),\Pi)\\
|D^2u^\epsilon|_{L^p(\Sigma)}\le C_3(1/\dist(\Sigma,\partial\Omega),\Pi, \Sigma, p)
\end{cases}
$$
for each open $\Sigma\subset\subset\Omega$, $p\in (n,\infty)$ and $\epsilon\in (0,1)$.  In particular, we note $\{u^\epsilon\}_{\epsilon\in (0,1)}\subset C^{1,1-n/p}_{\text{loc}}(\Omega)$ is bounded by Morrey's inequality. Consequently, there is a function
$w\in W^{2,p}_\text{loc}(\Omega)\cap W^{1,\infty}_\text{loc}(\Omega)$ and a sequence $\{\epsilon_k\}_{k\in\N}\subset(0,1)$ decreasing to $0$ as $k\rightarrow \infty$ such that $u^{\epsilon_k}$ converges to $w$ in $C^{1,1-n/p}(\Sigma)$ and weakly in $W^{2,p}(\Sigma)$ for each $\Sigma\subset\subset\Omega$. 
\par Now define 
$$
v(x):=
\begin{cases}
w(x), \quad x\in \Omega\\
\varphi(x), \quad x\in \partial\Omega
\end{cases}.
$$
By \eqref{uepsLinf}, $\underline{u}\le v\le \overline{u}$ on $\Omega$. As $\underline{u}, \overline{u}\in C(\overline{\Omega})$ and $\underline{u}\equiv\overline{u}$ on $\partial\Omega$, it follows that $v\in C(\overline{\Omega})$. 
We now claim that $v$ is a viscosity solution of the PDE \eqref{MainPDE}. By part $(i)$ of 
Theorem \ref{mainTHM}, we would then have $u\equiv v\in W^{2,p}_\text{loc}(\Omega)\cap W^{1,\infty}_\text{loc}(\Omega)$ for each $p\ge 1$. Moreover, as $H$ is assumed to be uniformly convex and $H(Du)\le 0$, we would additionally have $u\in W^{1,\infty}(\Omega)$.

% Subsolution 
\par 2. Let $x\in \Omega$ and suppose $v-\psi$ has a local maximum at $x_0$ where $\psi\in C^\infty(\Omega)$. We will show 
\begin{equation}\label{SubSolnINeq}
\max\{F(D^2\psi(x_0), x_0)-f(x_0), H(D\psi(x_0))\}\le 0.
\end{equation}
By adding $\frac{\rho}{2}|x-x_0|^2$ to $\psi$ and later sending $\rho\rightarrow 0^+$, we may assume that $v-\psi$ has a {\it strict} local maximum.  
As $u^{\epsilon_k}$ converges locally uniformly to $v$, there is a sequence $\Omega\ni x_k\rightarrow x_0$ such that $u^{\epsilon_k}-\psi$ 
has a local maximum at $x_k$. As $u^{\epsilon_k}$ is a classical solution of \eqref{penalized}, 
$$
F(D^2\psi(x_k),x_k)+\beta_{\epsilon_k}(H(D\psi(x_k))) \le f(x_k). 
$$
Given that $\beta_{\epsilon_k}\ge 0$, it follows $F(D^2\psi(x_k),x_k)\le f(x_k)$.  And after sending $k\rightarrow\infty$,
$$
F(D^2\psi(x_0), x_0)\le f(x_0).
$$
\par Recall that Lemma \ref{betaBound} implies 
$$
0\le \beta_{\epsilon_k}(H(D\psi(x_k)))=\beta_{\epsilon_k}(H(Du^{\epsilon_k}(x_k)))\le C
$$
for all sufficiently large $k\in \N$. By \eqref{betaAss}, it must also be that 
$$
H(D\psi(x_0))\le 0.
$$
Thus, inequality \eqref{SubSolnINeq} holds. 
% Supersolution 
\par 3. Conversely, assume that $v-\psi$ has a local minimum at $x_0$ where again $\psi\in C^\infty(\Omega)$. We claim
\begin{equation}\label{SupSolnINeq}
\max\{F(D^2\psi(x_0), x_0)-f(x_0), H(D\psi(x_0))\}\ge 0.
\end{equation}
Since we can subtract $\frac{\rho}{2}|x-x_0|^2$ from $\psi$ and later send $\rho\rightarrow 0^+$, we may assume $v-\psi$ has a {\it strict} local 
minimum. In this case, there is a sequence $\Omega\ni x_k\rightarrow x_0$ such that $u^{\epsilon_k}-\psi$ 
has a local minimum at $x_k$. As $u^{\epsilon_k}$ is a classical solution of \eqref{penalized}, 
\begin{equation}\label{SuperSolnk}
F(D^2\psi(x_k),x_k)+\beta_{\epsilon_k}(H(D\psi(x_k))) \ge f(x_k). 
\end{equation}
If $H(D\psi(x_0))\ge 0$, then \eqref{SupSolnINeq} clearly holds. Let us assume on the contrary that 
$$
H(D\psi(x_0))<0.
$$
Since $Du^\epsilon_k(x_k)\rightarrow Dv(x_0)=D\psi(x_0)$, 
$$
H(D\psi(x_k))<0
$$
for all sufficiently large $k$. Hence, $\beta_{\epsilon_k}(H(D\psi(x_k)))=0$ for all large $k$ and passing to the limit in \eqref{SuperSolnk}
gives
$$
F(D^2\psi(x_0), x_0)\ge f(x_0).
$$
This proves \eqref{SupSolnINeq}.

\par 4. So far, we have proved part $(ii)$ of Theorem \ref{mainTHM} under the additional
smoothness assumption \eqref{ExtraSmooth}. Now let $F$, $H$ and $f$ be as described in part $(ii)$ of statement of Theorem \ref{mainTHM}. We say a function $h:\Mn\rightarrow\R$ is integrable if when considered as a function of the $n^2$
variables $M_{11},\dots,M_{ij},\dots,M_{nn}$ it is integrable on $\R^{n^2}$ with respect to Lebesgue measure. In this case, we define the integral of $h$ to be
$$
\int_{\Mn} h(M)dM:=\int_{\R^{n^2}} h\left(M_{11},\dots,M_{ij},\dots,M_{nn}\right)\prod^n_{i,j=1}dM_{ij}.
$$
This allows us a convenient smoothing of $F=F(M,x)$ in the $M$ variable by a standard mollifier on $\Mn$. 

\par To this end, we select $\rho=\rho(M)$ that only depends on $|M|$ and satisfies 
$$
\begin{cases}
\rho\in C^\I(\Mn)\\
\rho\ge 0\\
\int_{\Mn} \rho(M)dM=1\\
\rho(M)=0,\quad |M|\ge 1
\end{cases}
$$
and define $\rho^{\delta}(M):=\delta^{-n^2}\rho(M/\delta)$ for $\delta>0$.  Likewise, we let $\varsigma$ denote a standard mollifier on $\R^n$ and set
$$
F^{\delta}(M,x):=\int_{B_\delta(0)}\int_{\Mn}\rho^{\delta}(N)\varsigma^\delta(y)F(M-N,x-y)dNdy, \quad (M,x)\in \Mn\times\Omega_\delta. 
$$
Here 
$$
\Omega_\delta:=\left\{x\in\Omega: \dist(x,\partial\Omega)>\delta\right\}
$$
and $\varsigma^\delta(y):=\delta^{-n}\varsigma(y/\delta)$.  It is readily verified that $F^{\delta}\in C^\I(\Mn\times\Omega_\delta)$ and satisfies \eqref{FUnifEll}, \eqref{Fconvex}, and \eqref{FXassump} with constant 
$(1+\delta)\Upsilon$ replacing $\Upsilon$. 

\par We may of course argue similarly to obtain smooth approximations of $H$ and $f$. Mollifying $H$, we obtain $H^\delta=\varsigma^\delta*H\in C^\I(\R^n)$ which converges to $H$ in $C^{1}_\text{loc}(\R^n)$ as $\delta\rightarrow 0^+$
and also satisfies \eqref{Hassump} at every $p\in \R^n$.  We also have that $f^\delta=\varsigma^\delta*f$ provides a smooth approximation of $f$ that converges locally uniformly to $f$ as $\delta\rightarrow 0^+$ on $\Omega$. Moreover, 
it routinely follows that $|Df^\delta|_{L^\infty(\Omega_\delta)}\le |Df|_{L^\infty(\Omega)}$.

 \par 5. As $\partial\Omega$ is smooth, there is $\delta_1>0$ and sufficiently small such that $\partial\Omega_\delta$ is smooth for $\delta\in (0,\delta_1)$; see Lemma 14.16 of \cite{GT}. Employing the penalty method as outlined in the previous section and passing to the limit as in parts 1-3 of this proof, we obtain a viscosity solution 
 $v^\delta\in W^{1,\infty}(\Omega_\delta)\cap W^{2,p}_\text{loc}(\Omega_\delta)$
 of the following boundary value problem
 \begin{equation}\label{deltaMainPDE}
 \begin{cases}
 \max\{F^\delta(D^2v,x)-f^\delta(x),H^\delta(Dv)\}=0, \quad x\in \Omega_\delta\\
 \hspace{2.4in} v=\underline{u},\quad x\in\partial\Omega_\delta
 \end{cases}
 \end{equation}
 for $\delta\in (0,\delta_1)$.  We define 
 $$
 u^\delta:=
 \begin{cases}
 v^\delta, \quad x\in\Omega_\delta\\
 \underline{u}, \quad x\in \overline{\Omega}\setminus \Omega_\delta
 \end{cases}
 $$
 and claim $u^\delta$ converges uniformly to $u$, the unique viscosity solution of \eqref{MainPDE} subject to the boundary condition \eqref{newBC}. 
 
 \par Observe that $u^\delta\in W^{1,\infty}(\Omega)$ and $H^\delta(Du^\delta(x))\le 0$ for almost every $x\in\Omega$.  It follows from the uniform convexity assumptions on $H$ \eqref{Hassump}  that the family of functions $\{u^\delta\}_{0<\delta<\delta_1}$ is uniformly 
 bounded and equicontinuous. Let  $\{\delta_k\}_{k\in \N}$ be a sequence of positive numbers tending to $0$ as $k\rightarrow \infty$. By the Arzel\`{a}-Ascoli theorem, the sequence of functions $\{u^{\delta_k}\}_{k\in \N}$ 
 has a subsequence $\{u^{\delta_{k_j}}\}_{j\in \N}$  converging uniformly to some $w\in C(\overline{\Omega})$ as $j\rightarrow\infty$. Letting $\delta=\delta_{k_j}$ in \eqref{deltaMainPDE} and passing to the limit as $j\rightarrow \infty$, we appeal to the stability of viscosity solutions under uniform convergence to 
 conclude that $w$ solves \eqref{MainPDE} and satisfies the boundary condition \eqref{newBC}.  Hence $w\equiv u$, and since the sequence $\{\delta_k\}_{k\in \N}$ was arbitrary, $u^\delta\rightarrow u$ uniformly on $\overline{\Omega}$.
 
 \par Now let $\Sigma\subset\subset\Omega$ be open and choose $\delta_2\in (0,\delta_1)$ to ensure $\Sigma\subset\subset\Omega_{\delta_2}$. Notice
 $$
0< \dist(\Sigma,\partial\Omega_{\delta_2})\le \dist(\Sigma,\partial\Omega_\delta)\le \dist(\Sigma,\partial\Omega)
 $$
 for $\delta\in (0,\delta_2)$. From the estimate in Corollary \ref{CorW2p} and our argument from part 1 of this proof, we have for $p\ge 1$
 \begin{equation}\label{D2udeltabound}
\left|D^2u^\delta\right|_{L^p(\Sigma)}\le C
 \end{equation}
where $C$ depends $n, p, \lambda,\Lambda,\theta,\Theta, \Sigma$ and
\begin{equation}\label{DeltaPiList}
\begin{cases}
(1+\delta)\Upsilon\\
\diam(\Omega_\delta)\\
H^\delta(0)\\
|DH^\delta(0)|\\
|F^\delta(O_n,\cdot)|_{L^\I(\Omega_\delta)}\\
|f^\delta|_{W^{1,\I}(\Omega_\delta)}\\
|\underline{u}|_{W^{1,\I}(\Omega_\delta)}\\
(\dist(\Sigma,\partial \Omega_\delta))^{-1}
\end{cases}.
\end{equation}
As each quantity in \eqref{DeltaPiList} is bounded uniformly in $\delta\in (0,\delta_2)$, \eqref{D2udeltabound} implies $\{u^\delta\}_{0<\delta<\delta_2}$ is bounded in $W^{2,p}(\Sigma)$. It follows that 
$u^\delta\rightharpoonup u$ in $W^{2,p}(\Sigma)$ and in particular that $u\in W^{2,p}(\Sigma)$. As $\Sigma$ was arbitrary,  $u\in W^{2,p}_{\text{loc}}(\Omega).$ 
\end{proof}

% Reference Section

\end{document}